\documentclass[12pt]{article}
\usepackage{amsthm}
\usepackage{amsmath}
\usepackage{enumerate}
\usepackage{amssymb}
\usepackage{latexsym}
\usepackage{amsfonts}
\usepackage{color}
\usepackage{mathrsfs}
\usepackage{epsfig}
\newtheorem{theorem}{\bf Theorem}[section]

\newtheorem{definition}[theorem]{\bf Definition}

\newtheorem{remark}[theorem]{\bf Remark}
\newtheorem{lemma}[theorem]{\bf Lemma}

\newsavebox{\savepar}

\begin{document}
	\title{Existence of infinitely many solutions for a nonlocal elliptic PDE involving singularity}
	\author{Sekhar Ghosh \&  Debajyoti Choudhuri~\footnote{Corresponding
			author: dc.iit12@gmail.com} \\
		\small{Department of Mathematics, National Institute of Technology Rourkela}\\
		\small{Emails: sekharghosh1234@gmail.com \& dc.iit12@gmail.com}}
	\date{}
	\maketitle
		\begin{abstract}
	\noindent In this article, we will prove the existence of infinitely many positive weak solutions to the following nonlocal elliptic PDE.
	\begin{align}
	(-\Delta)^s u&= \frac{\lambda}{u^{\gamma}}+ f(x,u)~\text{in}~\Omega,\nonumber\\
	u&=0~\text{in}~\mathbb{R}^N\setminus\Omega,\nonumber
	\end{align}
	where $\Omega$ is an open bounded domain in $\mathbb{R}^N$ with Lipschitz boundary, $N>2s$, $s\in (0,1)$, $\gamma\in (0,1)$. We will employ variational techniques to show the existence of infinitely many weak solutions of the above problem.\\
	{\bf keywords}: Elliptic PDE, Genus, PS condition, Mountain Pass Theorem.
\end{abstract}

\section{Introduction}
	We consider the following nonlocal problem involving singularity.
		\begin{align}\label{main1}
		\begin{split}
			(-\Delta)^s u&= \frac{\lambda}{u^{\gamma}}+ f(x,u)~\text{in}~\Omega,\\
		u&=0~\text{in}~\mathbb{R}^N\setminus\Omega,\\
		u&>0~\text{in}~\Omega,
		\end{split}
	\end{align}
	where, 
	$$(-\Delta)^s u(x)=-\frac{1}{2}\int_{\mathbb{R}^N}\frac{(u(x+y)+u(x-y)-2u(x))}{|x-y|^{N+2s}}dy, ~\text{for}~x\in\mathbb{R}^N,$$ $\lambda>0$, $s\in(0,1)$, $\gamma\in(0,1)$ and $\Omega$ be an open, bounded subset of $\mathbb{R}^N$, $N\geq 2$.
	
	\noindent The study of nonlocal elliptic PDEs is important to both, from the mathematical research as well as from the real world application, point of views. Some of the applications are in the probability theory, the obstacle problem, optimization, finance, phase transitions, soft thin films, conservation laws, minimal surfaces, material science and water waves. The application in probability theory, in particular to the  L\'{e}vy process, can be found in \cite{bertoin1996levy} and that in the field of finance, one can refer to \cite{cont2004financial}. For a fruitful note on the application of fractional Laplacian one may refer \cite{valdinoci2009long} and the references therein.
\noindent Recently, the study of singular elliptic PDE has drawn a great attention to many researchers. One of the earliest study on the existence and regularity of a weak solution was made by Lazer and Mckenna \cite{lazer1991singular} to the following problem. 
\begin{eqnarray}\label{lazer}
-\Delta u&=& \frac{p(x)}{u^\gamma}~\text{in}~\Omega,\nonumber\\
u&=&0~\text{in}~\partial\Omega,\\
u&>& 0~\text{in}~\Omega\nonumber,
\end{eqnarray}
where, $p:\Omega\rightarrow\mathbb{R}$ is a nonnegative bounded function. In \cite{lazer1991singular}, the authors proved that for a $C^{2+\gamma}$ boundary, the problem \eqref{lazer} has a $H_0^1(\Omega)$ solution iff $\gamma<3$ and if $\gamma>1$, the problem cannot have $C^1(\bar\Omega)$ solution. The following problem have been studied for existence, uniqueness and regularity of solutions for $p=2$ in \cite{fang2014existence} and for  $1<p<\infty$ in \cite{canino2017nonlocal}, where $0<s<1$.
\begin{eqnarray}\label{fang, canino nonlocal}
(-\Delta_p)^s u&=& \lambda\frac{ a(x)}{u^\gamma}~\text{in}~\Omega,\nonumber\\
u&=&0~\text{in}~\mathbb{R}^N\setminus\Omega,\\
u&>& 0~\text{in}~\Omega\nonumber,
\end{eqnarray}
where, $a:\Omega\rightarrow\mathbb{R}$ is a nonnegative bounded function. The author in \cite{fang2014existence}, guaranteed the existence of a unique $C^{2,\alpha}(\Omega),\,(0<\alpha<1)$ solution for $\lambda a(x)\equiv 1$. Canino et al. \cite{canino2017nonlocal}, had proved the existence and uniqueness of solution to the problem \eqref{fang, canino nonlocal} by dividing $\lambda$ into three cases $0<\lambda<1$, $\lambda=1$ and $\lambda>1$. A few more noteworthy study involving singularity for both Laplacian and fractional Laplacian operators can be found in \cite{boccardo2010semilinear, crandall1977dirichlet, ghosh2018singular, rosen1971minimum} and the references therein.\\ 
Multiplicity of solutions to the following type of problem has been widely studied by many authors, a few of them are in \cite{ghanmi2016nehari, ghosh2018multiplicity, mukherjee2016dirichlet,  saoudi2017critical} and the references therein.
\begin{eqnarray}\label{refer nonlocal}
(-\Delta_p)^s u&=& \lambda\frac{ a(x)}{u^\gamma}+f(x,u)~\text{in}~\Omega,\nonumber\\
u&=&0~\text{in}~\mathbb{R}^N\setminus\Omega,\\
u&>& 0~\text{in}~\Omega\nonumber.
\end{eqnarray}
Here $N>ps$, $M\geq0$, $a:\Omega\rightarrow\mathbb{R}$ is a nonnegative bounded function. The authors in \cite{ghosh2018multiplicity, saoudi2017critical}, have used a variational technique to guarantee the existence of multiple solutions. Nehari manifold method has been used to prove the multiplicity result in \cite{ghanmi2016nehari, mukherjee2016dirichlet}. In most of these studies, the authors obtained two distinct weak solutions.\\
The existence results of infinitely many solutions to both Laplacian and fractional Laplacian with a nonsingular, nonlinear data have been studied widely with Dirichlet boundary condition. In most of these studies the authors proved the existence result with the help of the symmetric Mountain Pass Theorem \cite{heinz1987free, kajikiya2005critical}. One of the earliest attempt to show the existence of infinitely many solutions was made by Ambrosetti and Rabinowitz \cite{ambrosetti1973dual} to the following problem.
\begin{align}\label{problem ambro}
\begin{split}
-\Delta u&= f(x,u)~\text{in}~\Omega,\\
u&=0~\text{in}~\partial\Omega,\\
u&>0~\text{in}~\Omega,
\end{split}
\end{align}
where, $f$ is superlinear but subcritical near infinity. The authors in [19] have used the symmetric Mountain Pass Theorem to guarantee the existence. A few more similar type of studies made are in \cite{kajikiya2005critical, liu2015clark, wang2001nonlinear} and the references therein.\\
Recently, Gu et al. \cite{gu2018motivation} has guaranteed the existence of infinitely many solutions to a nonlocal problem of the following type with a sublinear growth of $f$.
	\begin{align}\label{problem lambda =0}
\begin{split}
(-\Delta)^s u&= f(x,u)~\text{in}~\Omega,\\
u&=0~\text{in}~\mathbb{R}^N\setminus\Omega,\\
u&>0~\text{in}~\Omega.
\end{split}
\end{align}
 For further details to the problem \eqref{problem lambda =0}, we refer the readers to \cite{binlin2015superlinear, bisci2016nontrivial, dipierro2017fractional, servadei2013variational, servadei2015brezis} and the references therein.\\
  In the literature the study to obtain infinitely many solutions for the problems of the type \eqref{problem ambro}, \eqref{problem lambda =0}, the authors have considered either a sublinear or a superlinear growth on $f$. To our knowledge, the study of the problem \eqref{main1} is very new to the literature due to the presence of a singular term $u^{-\gamma}$. Motivated from \cite{gu2018motivation}, we will prove the existence of infinitely many weak solutions to the problem \eqref{main1}. We assume the following growth conditions on $f$.
  \begin{itemize}
  	\item[(A1)] $f\in C(\Omega\times\mathbb{R}, \mathbb{R})$ and there exists a $\delta>0$ such that $\forall\,x\in\Omega$ and $|t|\leq\delta,$  $f(x,-t)=-f(x,t).$
  	\item[(A2)] $\lim\limits_{t\rightarrow0}\frac{f(x,t)}{t}=+\infty$ uniformly on $\Omega.$
  	\item[(A3)] There exists $r>0$ and $\alpha\in(1-\gamma, 2)$ such that $\forall,\,x\in\Omega$ and $|t|\leq r$, $tf(x,t)\leq\alpha F(x,t)$, where $F(x,t)=\int_{0}^{t}f(x,\tau) d\tau.$
  \end{itemize}
  Prior to stating the main theorem, we will define the necessary function spaces and the associated notations. Consider the space $(X, \|.\|_X)$ which is defined as
  \begin{eqnarray}
  X&=&\left\{u:\mathbb{R}^N\rightarrow\mathbb{R}~\text{is measurable}, u|_{\Omega}\in L^2(\Omega) ~\text{and}~\frac{|u(x)-u(y)|}{|x-y|^{\frac{N+2s}{2}}}\in L^{2}(Q)\right\}\nonumber
  \end{eqnarray}
  equipped with the Gagliardo norm 
  \begin{eqnarray}
  \|u\|_X&=&\|u\|_{2}+\left(\int_{Q}\frac{|u(x)-u(y)|^2}{|x-y|^{N+2s}}dxdy\right)^{\frac{1}{2}}.\nonumber
  \end{eqnarray}
  where $\Omega\subset\mathbb{R}^N$, $Q=\mathbb{R}^{2N}\setminus((\mathbb{R}^N\setminus\Omega)\times(\mathbb{R}^N\setminus\Omega))$. Here, $\|u\|_{2}$ refers to the $L^2$-norm of $u$. A frequently used space in this article will be the subspace $X_0$ of $X$ defined as
  \begin{eqnarray}
  X_0&=&\left\{u\in X: u=0 ~\text{a.e. in}~ \mathbb{R}^N\setminus\Omega\right\}\nonumber
  \end{eqnarray}
  equiped with the norm
  \begin{eqnarray}
  \|u\|&=&\left(\int_{Q}\frac{|u(x)-u(y)|^2}{|x-y|^{N+2s}}dxdy\right)^{\frac{1}{2}}.\nonumber
  \end{eqnarray}
  The space $(X_0, \|.\|)$ is a Hilbert space \cite{servadei2012mountain}. The best Sobolev constant is defined as 
  \begin{equation}\label{sobolev const}
  S=\underset{u\in X_0\setminus\{0\}}{\inf}\cfrac{\int_{Q}\frac{|u(x)-u(y)|^2}{|x-y|^{N+2s}}dxdy}{\left(\int_\Omega|2_s^*|dx\right)^{\frac{2}{2_s^*}}}
  \end{equation}

\noindent We now define a weak solution to the problem \eqref{main1}.	
\begin{definition}\label{weak main}
	A function $u\in X_0$ is a weak solution to the problem \eqref{main1}, if $u>0,$ $\phi u^{-\gamma}\in L^1(\Omega)$ and
	\begin{equation}
	\int_{Q}\frac{(u(x)-u(y))(\phi(x)-\phi(y))}{|x-y|^{N+2s}}dxdy-\int_{\Omega}\left(\frac{\lambda}{u^{\gamma}}+f(x,u)
	\right)\phi dx=0,~\forall\,\phi\in X_0.
	\end{equation}
\end{definition}

\noindent The energy functional $I:X_0\rightarrow(-\infty, \infty]$ associated with the problem \eqref{main1} is defined as
\begin{equation}\label{energy main}
I(u)=\frac{1}{2}\int_{Q}\frac{|u(x)-u(y)|^2}{|x-y|^{N+2s}}dxdy-\frac{\lambda}{1-\gamma}\int_{\Omega}u^{1-\gamma}dx-\int_{\Omega}F(x, u)dx
\end{equation}

\noindent We will now state our main result
		\begin{theorem}\label{main thm}
		Let the assumptions $(A1)-(A3)$ hold, then there exists $\Lambda<\infty$ and for every $\lambda\in(0, \Lambda)$, the problem \eqref{main1} has a sequence of nonnegative weak solutions $\{u_n\}\subset X_0\cap L^{\infty}(\Omega)$ such that $I(u_n)<0,$ $I(u_n)\rightarrow0^-$ and $u_n\rightarrow0$ in $X_0.$
	\end{theorem}
\noindent We will use the symmetric Mountain Pass Theorem version due to Clark \cite{clark1972variant}, to guarantee the existence of distinct infinitely many weak solutions. There exists two versions of the symmetric Mountain Pass Theorem. One of them gives a sequence of critical values diverging to infinity for the superlinear data. The other one provides a sequence of critical values converging to zero for the sublinear data. In the present article, we will use the theorem for the sublinear case. To state the symmetric Mountain Pass Theorem, we will define the notion of genus, which will be used to prove our main Theorem \ref{main thm}.
\begin{definition}[\bf{Genus}]\label{genus}
Let $X$ be a Banach space and $A\subset X$. A set $A$ is said to be symmetric if $u\in A$ implies $-u\in A$. Let $A$ be a close, symmetric subset of $X$ such that $0\notin A$. We define a genus $\gamma(A)$ of $A$ by the smallest integer $k$ such that there exists an odd continuous mapping from $A$ to $\mathbb{R}^{k}\setminus\{0\}$. We define $\gamma(A)=\infty$, if no such $k$ exists.
\end{definition}
\noindent Let us consider the following set,\vspace{-.3cm} $$\Gamma_n=\{A_n\subset X: A_n~\text{is closed, symmetric and}~ 0\notin A_n~\text{such that the genus}~ \gamma(A_n)\geq n\}.$$
The following version of the symmetric Mountain Pass Theorem has been taken from \cite{kajikiya2005critical}.
	\begin{theorem}\label{sym mountain}
		Let $X$ be an infinite dimensional Banach space and $\tilde{I}\in C^1(X,\mathbb{R})$ satisfies the following
		\begin{itemize}
			\item[(i)] $\tilde{I}$ is even, bounded below, $\tilde{I}(0)=0$ and $\tilde{I}$ satifies the $(PS)_c$ condition.
			\item[(ii)] For each $n\in\mathbb{N}$, there exists an $A_n\in\Gamma_n$ such that $\sup\limits_{u\in A_n}\tilde{I}(u)<0.$ 
		\end{itemize}
		Then for each $n\in\mathbb{N}$, $c_n=\inf\limits_{A\in \Gamma_n}\sup\limits_{u\in A}\tilde{I}(u)<0$ is a critical value of $\tilde{I}.$
	\end{theorem}

\noindent In order to apply the symmetric Mountain Pass Theorem, we modify the problem \eqref{main1} to
		\begin{align}\label{main2}
	\begin{split}
	(-\Delta)^s u&= \lambda\cfrac{sign(u)}{|u|^{\gamma}}+ f(x,u)~\text{in}~\Omega,\\
	u&=0~\text{in}~\mathbb{R}^N\setminus\Omega,
	\end{split}
	\end{align}
	
	\noindent We define the energy functional $J:X_0\rightarrow(-\infty, \infty]$ associated with the problem \eqref{main2} as
	\begin{equation}\label{energy modified}
	J(u)=\frac{1}{2}\int_{Q}\frac{|u(x)-u(y)|^2}{|x-y|^{N+2s}}dxdy-\frac{\lambda}{1-\gamma}\int_{\Omega}|u|^{1-\gamma}dx-\int_{\Omega}F(x, u)dx
	\end{equation}
	We now give the definition of a weak solution to the problem \eqref{main2}.
		\begin{definition}\label{weak modified}
		A function $u\in X_0$ is a weak solution of \eqref{main2}, if $\phi |u|^{-\gamma}\in L^1(\Omega)$ and
		\begin{equation}
		\int_{Q}\frac{(u(x)-u(y))(\phi(x)-\phi(y))}{|x-y|^{N+2s}}dxdy-\int_{\Omega}\left(\frac{\lambda sign(u)}{|u|^{\gamma}}+f(x,u)
		\right)\phi dx=0,
		\end{equation}
		for all $\phi\in X_0.$
	\end{definition}

	\noindent It is easy to see that, if $u$ is a weak solution to the problem \eqref{main2} and $u>0$ a.e., then $u$ is also a weak solution to the problem \eqref{main1}. We use a {\it cutoff} technique given in \cite{clark1972variant} to guarantee the existence of infinitely many positive weak solutions to the problem \eqref{main2}. Let us choose $l$ to be small, such that $0<l\leq\frac{1}{2} \min \{\delta, r\},$ where $\delta$ and $r$ are same as in the assumptions on $f.$ Let us define a $C^1$ function $\xi:\mathbb{R}\rightarrow\mathbb{R}^+$ such that $0\leq\xi(t)\leq1$ and 
	$$\xi(t)=\begin{cases}
	1, ~\text{if}~ |t|\leq l\\
	\xi ~\text{is decreassing, if}~ l\leq t\leq 2l\\
	0,~\text{if}~ |t|\geq 2l.
	\end{cases}$$
	Since our main objective is to prove the existence of positive solutions, we will define $f(x, t)=0$ for $t\leq0$. Let us consider the following {\it cutoff} problem.
	
	\begin{align}\label{main3}
	\begin{split}
	(-\Delta)^s u&= \lambda\cfrac{sign(u)}{|u|^{\gamma}}+ \tilde{f}(x,u)~\text{in}~\Omega,\\
	u&=0~\text{in}~\mathbb{R}^N\setminus\Omega,
	\end{split}
	\end{align}
	where, $$\tilde{f}(x, u)=\begin{cases}
	f(x, u)\xi(u), ~\text{for}~u\geq0\\
	0, ~\text{for}~u\leq0.
	\end{cases}$$
	One can easily see that if $u$ is a weak solution to \eqref{main3} with $\|u\|_{\infty}\leq l$, then $u$ is also a weak solution to \eqref{main2}. We will investigate the existence of infinitely many weak solutions to the problem \eqref{main3}. Moreover, to achieve our goal we will prove that $\|u\|_{\infty}\leq l$ and the solutions to \eqref{main3} are positive.
	
	\noindent The energy functional $\tilde{I}:X_0\rightarrow(-\infty, \infty]$ associated with the problem \eqref{main3} is defined as
	\begin{equation}\label{energy cutoff}
	\tilde{I}(u)=\frac{1}{2}\int_{Q}\frac{|u(x)-u(y)|^2}{|x-y|^{N+2s}}dxdy-\frac{\lambda}{1-\gamma}\int_{\Omega}|u|^{1-\gamma}dx-\int_{\Omega}\tilde{F}(x, u)dx
	\end{equation}
	A weak solution to the problem \eqref{main3} is given by the following definition.
	\begin{definition}\label{weak cutoff}
		A function $u\in X_0$ is a weak solution of \eqref{main3}, if $\phi |u|^{-\gamma}\in L^1(\Omega)$ and
	\begin{equation}
	\int_{Q}\frac{(u(x)-u(y))(\phi(x)-\phi(y))}{|x-y|^{N+2s}}dxdy-\int_{\Omega}\left(\frac{\lambda sign(u)}{|u|^{\gamma}}+\tilde{f}(x,u)
	\right)\phi dx=0,
	\end{equation}
	for all $\phi\in X_0.$
   \end{definition}
\noindent Henceforth, a weak solution will be referred to as a solution.
	
\section{Existence and Multiplicity of solutions}
We begin this section by proving that 
$$\Lambda=\inf\{\lambda>0:~\text{The problem \eqref{main1} has no weak solution} \}.$$ is a finite, nonnegative real number.
\begin{lemma}
Assume $0<\gamma<1$ and $(A1)-(A3)$ holds. Then $0\leq\Lambda<\infty$.
\end{lemma}
\begin{proof}
It is clear that $\Lambda\geq0$ from its definition. Let $\lambda_1$ be the principal eigenvalue of the fractional Laplacian operator $(-\Delta)^s$ in $\Omega$ and let $\phi_1>0$ be the associated eigenfunction \cite{brasco2016second}. Therefore, we have
\begin{align}
(-\Delta)^s\phi_1&=\lambda_1\phi_1~\text{in}~\Omega,\nonumber\\
\phi_1&>0~\text{in}~\Omega,\nonumber\\
\phi&=0~\text{in}~\mathbb{R}^N\setminus\Omega. 
\end{align}      
By choosing $\phi_1$ as a test function in the Definition \ref{weak main}, we get
\begin{align}\label{contradiction}
\lambda_1\int_{\Omega}u\phi_1dx&=\int_{\Omega}(-\Delta)^s\phi_1 u dx\nonumber\\
&=\int_{\Omega}\left(\frac{\lambda}{u^{\gamma}}+f(x, u)\right)\phi_1dx.
\end{align}
 Let us now choose, any arbitrary constant $\tilde{\varLambda}$ such that $\tilde{\Lambda}t^{-\gamma}+f(x, t)>2\lambda_1 t$ $\forall t>0$. This contradicts to the equation \eqref{contradiction}. Hence, we get $\Lambda<\infty$.
\end{proof}
\begin{remark}
	In fact, we will finally prove that $\Lambda>0$.
\end{remark}
\noindent We will now prove the following Lemma, to obtain one of  the hypothesis of the symmetric Mountain Pass Theorem.
 \begin{lemma}\label{lemma ps}
 	The functional $\tilde{I}$ is bounded below and satisfies $(PS)_c$ condition.
 \end{lemma}
\begin{proof}
	 By using the definition of $\xi$ and H\"{o}lder's inequality, we get 
	 \begin{align*}
	 \tilde{I}(u)&\geq\frac{1}{2}\int_{Q}\frac{|u(x)-u(y)|^2}{|x-y|^{N+2s}}dxdy-\lambda C\|u\|^{1-\gamma}-C_1\\
	 &\geq\frac{1}{2}\|u\|^2-\lambda C\|u\|^{1-\gamma}-C_1
	 \end{align*}
	 where, $C$, $C_1$ are non negative constants. This implies that $\tilde{I}$ is coercive and bounded below in $X_0$. Let $\{u_n\}\subset X_0$ be a Palais Smale sequence for $\tilde{I}$. Then $\{u_n\}$ is bounded in $X_0$ due to the coerciveness of $\tilde{I}.$ Therefore, we may assume, $u_n\rightharpoonup u$ in $X_0$ upto a subsequence. Thus, we have
	 \begin{equation}\label{convergence weak}
	 \int_{Q}\frac{(u_n(x)-u_n(y))(\phi(x)-\phi(y))}{|x-y|^{N+2s}}dxdy\longrightarrow\int_{Q}\frac{(u(x)-u(y))(\phi(x)-\phi(y))}{|x-y|^{N+2s}}dxdy
	 \end{equation} 
	 for all $\phi\in X_0.$
	 By the embedding result \cite{servadei2012mountain}, we can assume
	 \begin{align}
	 u_n&\longrightarrow u ~\text{in}~ L^p(\Omega),\label{embed strong}\\
	  u_n(x)&\longrightarrow u(x) ~\text{a.e.}~ L^p(\Omega).\label{embed pointwise}
	 \end{align}
	 Therefore, from Lemma A.1 \cite{willem1997minimax}, we get that there exists $g\in L^p(\Omega)$ such that
	 \begin{equation}\label{appendeix A1}
	 |u_n(x)|\leq g(x) ~\text{a.e. in}~ \Omega, \forall\,n\in\mathbb{N}.
	 \end{equation}
	 Hence, by using \eqref{embed strong}, \eqref{embed pointwise}, \eqref{appendeix A1} and the Lebesgue dominated convergence theorem, we get
	 \begin{equation}\label{convergence f tilla}
	 \int_{\Omega}\tilde{f}(x,u_n)udx\rightarrow\int_{\Omega}\tilde{f}(x,u)udx ~\text{and}~ \int_{\Omega}\tilde{f}(x,u_n)u_ndx\rightarrow\int_{\Omega}\tilde{f}(x,u)udx
	 \end{equation}
	 Again, on using the H\"{o}lder's inequality and passing to the limit $n\rightarrow\infty$, we get
	 \begin{align}
	 \begin{split}
	 \int_{\Omega}u_n^{1-\gamma}dx&\leq\int_{\Omega}u^{1-\gamma}dx+\int_{\Omega}|u_n-u|^{1-\gamma}dx\\
	 &\leq\int_{\Omega}u^{1-\gamma}dx+C\|u_n-u\|_{L^2(\Omega)}^{1-\gamma}\\
	 &=\int_{\Omega}u^{1-\gamma}dx +o(1)
	 \end{split}
	 \end{align}
	 Similarly, we have
	 \begin{align}
	 \begin{split}
	 \int_{\Omega}u^{1-\gamma}dx&\leq\int_{\Omega}u_n^{1-\gamma}dx+\int_{\Omega}|u_n-u|^{1-\gamma}dx\\
	 &\leq\int_{\Omega}u_n^{1-\gamma}dx+C\|u_n-u\|_{L^2(\Omega)}^{1-\gamma}\\
	 &=\int_{\Omega}u_n^{1-\gamma}dx +o(1)
	 \end{split}
	 \end{align}
	 Therefore,
	 \begin{equation}\label{convergence singular}
	 \int_{\Omega}u_n^{1-\gamma}dx=\int_{\Omega}u^{1-\gamma}dx+o(1)
	 \end{equation}
	 Now, since $\langle\tilde{I}(u_n),u_n\rangle\rightarrow0$, we have
	 \begin{equation}\label{convergence I tilla}
	 \int_{Q}\frac{|u_n(x)-u_n(y)|^2}{|x-y|^{N+2s}}dxdy-\lambda\int_{\Omega}|u_n|^{1-\gamma}dx-\int_{\Omega}\tilde{f}(x, u_n)u_ndx\rightarrow0
	 \end{equation}
	 Therefore, by \eqref{convergence f tilla}, \eqref{convergence singular} and \eqref{convergence I tilla}, we get
	 \begin{equation}\label{convergence u_n}
	 \int_{Q}\frac{|u(x)-u(y)|^2}{|x-y|^{N+2s}}dxdy\rightarrow\lambda\int_{\Omega}|u|^{1-\gamma}dx-\int_{\Omega}\tilde{f}(x, u)udx
	 \end{equation}
	 Moreover,
	 \begin{align}
	 \begin{split}
	\langle\tilde{I}(u_n),u\rangle&=\int_{Q}\frac{(u_n(x)-u_n(y))(u(x)-u(y))}{|x-y|^{N+2s}}dxdy\\
	&\hspace{2cm}-\lambda\int_{\Omega}sign(u_n)|u_n|^{-\gamma}udx-\int_{\Omega}\tilde{f}(x, u_n)udx
	 \end{split}
	 \end{align}
	 Note that
	 \begin{align*}
	 \langle\tilde{I}(u_n),u\rangle&\longrightarrow0, ~\text{as}~ n\rightarrow\infty.
	 \end{align*}
	 Hence, taking $\phi=u$ in \eqref{convergence weak} and by using \eqref{convergence singular}-\eqref{convergence u_n}, we get
	 \begin{equation}\label{convergence norm u}
	 \int_{Q}\frac{|u(x)-u(y)|^2}{|x-y|^{N+2s}}dxdy=\lambda\int_{\Omega}|u|^{1-\gamma}dx+\int_{\Omega}\tilde{f}(x, u)udx
	 \end{equation}
	 Therefore, we conclude $\|u_n\|\rightarrow\|u\|$ and this completes the proof.
\end{proof}
\noindent We will now prove a Lemma which will guarantee that, for each $n\in\mathbb{N}$, the set $\Gamma_n\neq\phi$, where $\phi$ is the empty set.
\begin{lemma}\label{lemma genus}
	For any $n\in\mathbb{N}$, there exists a closed, symmetric subset $A_n\subset X_0$ with $0\notin A_n$ such that the genus $\gamma(A_n)\geq n$ and $\sup\limits_{u\in A_n}\tilde{I}(u)<0.$
\end{lemma}
\begin{proof}
	We will first guarantee an existence of a closed, symmetric subset $A_n$ over every finite dimensional subspace of $X_0$ such that $\gamma(A_n)\geq n.$ Let $X_k$ be a subspace of $X_0$ such that $\dim (X_k)=k.$ We know that every norm over a finite dimensional norm linear space are equivalent. Therefore, there exists a positive constant $L=L(k)$ such that $\|u\|\leq L\|u\|_{L^2(\Omega)}$ for all $u\in X_k.$\\
	{\bf Claim:}
		There exists a positive constant $R$ such that 
		\begin{equation}\label{claim genus}
		\frac{1}{2}\int_{\Omega}|u|^2dx\geq\int_{\{|u|> l\}}|u|^2dx,~\forall\,u\in X_k ~\text{such that}~ \|u\|\leq R.
		\end{equation}
The proof is by contradiction. Let $\{u_n\}$ be a sequence in $X_k\setminus\{0\}$ such that $u_n\rightarrow0$ in $X_0$ and
	\begin{equation}
\frac{1}{2}\int_{\Omega}|u_n|^2dx<\int_{\{|u_n|> l\}}|u_n|^2dx.
\end{equation}
Choose, $v_n=\frac{u_n}{\|u_n\|_{L^2(\Omega)}}.$ Then 
	\begin{equation}\label{claim contra eqn}
\frac{1}{2}<\int_{\{|u_n|> l\}}|v_n|^2dx.
\end{equation}
Now, since $\dim(X_k)=k$, we can assume $v_n\rightarrow v$ in $X_0$ upto a subsequence. Therefore, $v_n\rightarrow v$ also in $L^2(\Omega).$ Further, observe that, $$m\{x\in\Omega: |u_n|>l\}\rightarrow0 ~\text{as}~ n\rightarrow\infty$$ since, $u_n\rightarrow0$ in $X_0$. This is a contradiction to \eqref{claim contra eqn}. Therefore, the claim is proved. Now, from the assumption $(A2)$, we can choose $0<l\leq1$ such that, $$\tilde{F}(x,t)=F(x,t)\geq2L^2t^2, ~\forall\, (x,t)\in\Omega\times[0,l].$$
Hence, for all $u\in X_k\setminus\{0\}$ such that $\|u\|\leq R$ and by using \eqref{claim genus}, we get
\begin{align*}
\tilde{I}(u)&\leq\frac{1}{2}\int_{Q}\frac{|u(x)-u(y)|^2}{|x-y|^{N+2s}}dxdy-\frac{\lambda}{1-\gamma}\int_{\Omega}|u|^{1-\gamma}dx-\int_{\{|u|\leq l\}}\tilde{F}(x, u)dx\\
&\leq\frac{1}{2}\|u\|^2-\frac{\lambda}{1-\gamma}\int_{\Omega}|u|^{1-\gamma}dx-2L^2\int_{\{|u|\leq l\}}|u|^2dx\\
&=\frac{1}{2}\|u\|^2-\frac{\lambda}{1-\gamma}\int_{\Omega}|u|^{1-\gamma}dx-2L^2\left(\int_{\Omega}|u|^2dx-\int_{\{|u|> l\}}|u|^2dx\right)\\
&\leq\frac{1}{2}\|u\|^2-\frac{\lambda}{1-\gamma}\int_{\Omega}|u|^{1-\gamma}dx-L^2\int_{\Omega}|u|^2dx\\
&\leq-\frac{1}{2}\|u\|^2-\frac{\lambda}{1-\gamma}\int_{\Omega}|u|^{1-\gamma}dx\\
&<0.
\end{align*}
\noindent Let us now choose, $0<\rho\leq R$ and $A_n=\{u\in X_n: \|u\|=\rho\}$. This serves the purpose of showing that $\Gamma_n\neq\phi$. Since $A_n$ is symmetric, closed with $\gamma(A_n)\geq n$ such that $\sup\limits_{u\in A_n}\tilde{I}(u)<0.$ This completes the proof.
\end{proof}
 \noindent The following Lemmas will be proved to guarantee the boundedness of the solutions to the problem \eqref{main3}.

 \begin{lemma}\label{bounded l1}
 	Let $g:\mathbb{R}\rightarrow\mathbb{R}$ be a convex $C^1$ function. Then for every $a, b, A, B\in\mathbb{R}$ with $A,B>0$ the following inequality holds.
 	\begin{equation}
 	(g(a)-g(b))(A-B)\leq (a-b)(Ag'(a)-Bg'(b))
 	\end{equation}
 \end{lemma}
\begin{proof}
	Since, $g$ is convex, we have
	\begin{align}\label{bounded l1 convex}
	g(b)-g(a)&\geq g'(a)(b-a) ~\text{and}~
	g(a)-g(b)\geq g'(b)(a-b)
	\end{align}
	Therefore, using \eqref{bounded l1 convex}, we get
	\begin{align*}
	(a-b)\left[Ag'(a)-Bg'(b)\right]&= A(a-b)g'(a)-B(a-b)g'(b)\\
	&\geq A\left[g(a)-g(b)\right]-B\left[g(a)-g(b)\right]\\
	&=(A-B)(g(a)-g(b))
	\end{align*}
\end{proof}

 \begin{lemma}\label{bounded l2}
	Let $h:\mathbb{R}\rightarrow\mathbb{R}$ be an increasing function, then for $a, b, \tau\in\mathbb{R}$ with $\tau\geq 0$ we have
	\begin{equation}
	[H(a)-H(b)]^2\leq (a-b)(h(a)-h(b))
	\end{equation}
	where, $H(t)=\int_0^t \sqrt{h'(\tau)}d\tau$, for $t\in\mathbb{R}.$
\end{lemma}

\begin{proof}
	\begin{align*}
		(a-b)(h(a)-h(b))&=(a-b)\int_{b}^{a}h'(\tau)d\tau\\
		&=(a-b)\int_{b}^{a}(H'(\tau))^2d\tau\\
		&\geq\left(\int_{b}^{a}H'(\tau)d\tau\right)^2 ~\text{by Jensen's inequality}\\
		&=[H(a)-H(b)]^2
	\end{align*}
	This completes the proof
\end{proof}\newpage
\begin{lemma}\label{bounded}
	Let $u\in X_0$ be a positive weak solution to the problem in \eqref{main3}, then $u\in L^{\infty}(\Omega).$
\end{lemma}
\begin{proof}
	We follow the steps from Brasco and Parini \cite{brasco2016second}. For every small $\epsilon>0,$ let us define the smooth function  
	\begin{equation*}
	g_{\epsilon}(t)=(\epsilon^2+t^2)^{\frac{1}{2}}
	\end{equation*}
	Observe that the function $g_{\epsilon}$ is convex and Lipschitz. For all positive $\psi\in C_c^{\infty}(\Omega)$, we take $\phi=\psi g'_{\epsilon}(u)$ to be the test function in \eqref{main3}. By choosing $a=u(x), b=u(y), A=\psi(x)$ and $B=\psi(y)$ in Lemma \ref{bounded l1}, we have
	\begin{align}\label{bound est 1}
	\int_{Q}\cfrac{(g_{\epsilon}(u(x))-g_{\epsilon}(u(y)))(\psi(x)
		-\psi(y))}{|x-y|^{N+2s}}dxdy\leq\int_\Omega\left(|\lambda u^{-\gamma}+\tilde{f}(x, u)|\right)|g'_{\epsilon}(u)|\psi dx
	\end{align}
	The function $g_{\epsilon}(t)\rightarrow|t|$ as $t\rightarrow0$ and hence $|g'_{\epsilon}(t)|\leq1.$ Therefore, on using Fatou's Lemma and passing the limit $\epsilon\rightarrow0$ in \eqref{bound est 1}, we obtain
	\begin{align}\label{bound est 2}
	\int_{Q}\cfrac{(|u(x)|-|u(y)|)(\psi(x)
		-\psi(y))}{|x-y|^{N+2s}}dxdy\leq\int_\Omega\left(|\lambda u^{-\gamma}+\tilde{f}(x, u)|\right)\psi dx
	\end{align}
	for all $\psi\in C_c^{\infty}(\Omega)$ with $\psi>0.$
	The inequality \eqref{bound est 2} remains true for all $\psi\in X_0$ with $\psi\geq0.$ We define the {\it cutoff} function $u_k=\min\{(u-1)^+, k\}\in X_0$ for $k>0.$ Now for any given $\beta>0$ and $\delta>0$, we choose $\psi=(u_k+\delta)^{\beta}-\delta^{\beta}$ as the test function in \eqref{bound est 2} and get
	\begin{align}\label{bound est 3}
	\int_{Q}&\cfrac{(|u(x)|-|u(y)|)((u_k(x)+\delta)^{\beta}
		-(u_k(y)+\delta)^{\beta})}{|x-y|^{N+2s}}dxdy\nonumber\\
	&\hspace{1cm}\leq\int_\Omega\left(|\lambda u^{-\gamma}+\tilde{f}(x, u)|\right)\left((u_k+\delta)^{\beta}-\delta^{\beta}\right) dx
	\end{align}
	Now applying the Lemma \ref{bounded l2} to the function $h(u)=(u_k+\delta)^{\beta},$ we get
	\begin{align*}\label{bound est 4}
	\begin{split}
	&\int_{Q}\cfrac{|((u_k(x)+\delta)^{\frac{\beta+1}{2}}
		-(u_k(y)+\delta)^{\frac{\beta+1}{2}})|^2}{|x-y|^{N+2s}}dxdy\\
	&\leq\frac{(\beta+1)^2}{4\beta}\int_{Q}\cfrac{(|u(x)|-|u(y)|)((u_k(x)+\delta)^{\beta}
		-(u_k(y)+\delta)^{\beta})}{|x-y|^{N+2s}}dxdy\\
	&\leq\frac{(\beta+1)^2}{4\beta}\int_\Omega\left(|\lambda u^{-\gamma}+\tilde{f}(x, u)|\right)\left((u_k+\delta)^{\beta}-\delta^{\beta}\right) dx\\
	&\leq\frac{(\beta+1)^2}{4\beta}\int_\Omega\left(|\lambda u^{-\gamma}|+|\tilde{f}(x, u)|\right)\left((u_k+\delta)^{\beta}-\delta^{\beta}\right) dx\\
	&=\frac{(\beta+1)^2}{4\beta}\int_{\{u\geq1\}}\left(|\lambda u^{-\gamma}|+|\tilde{f}(x, u)|\right)\left((u_k+\delta)^{\beta}-\delta^{\beta}\right) dx\\	
	&\leq\frac{(\beta+1)^2}{4\beta}\int_{\{u\geq1\}}\left(|\lambda|+(|c_1|+|c_2||u|^{\alpha})\right)\left((u_k+\delta)^{\beta}-\delta^{\beta}\right) dx\\
	\end{split}
	\end{align*}
	\begin{align}
	\begin{split}
	&\leq C_1\frac{(\beta+1)^2}{4\beta}\int_{\{u\geq1\}}\left(1+|u|^{\alpha}\right)\left((u_k+\delta)^{\beta}-\delta^{\beta}\right) dx\\
	&\leq 2C_1\frac{(\beta+1)^2}{4\beta}\int_{\{u\geq1\}}|u|^{\alpha}\left((u_k+\delta)^{\beta}-\delta^{\beta}\right) dx\\
	&\leq C\frac{(\beta+1)^2}{4\beta}|u|_{2_s^*}^{\alpha}|(u_k+\delta)^{\beta}|_q 
	\end{split}
	\end{align}
	where, $q=\frac{2_s^*}{2_s^*-\alpha}$ and $C=\max\{1,|\lambda|\}.$
Now by using the Sobolev inequality for $X_0$, given by \cite{di2012hitchhikers}, we get
\begin{align}\label{bound est 5}
\int_{Q}\cfrac{|((u_k(x)+\delta)^{\frac{\beta+1}{2}}
	-(u_k(y)+\delta)^{\frac{\beta+1}{2}})|^2}{|x-y|^{N+2s}}dxdy\geq C_{N, s}|(u_k+\delta)^{\frac{\beta+1}{2}}-\delta^{\frac{\beta+1}{2}}|_{2_s^*}^2
\end{align}
where, $C_{N,s}$ is the embedding constant. The triangle inequality and $(u_k+\delta)^{\beta+1}\geq\delta(u_k+\delta)^{\beta}$ implies
\begin{align}\label{bound est 6}
\left[\int_{\Omega}\left((u_k+\delta)^{\frac{\beta+1}{2}}
	-\delta^{\frac{\beta+1}{2}}\right)^{2_s^*}dx\right]^{\frac{2}{2_s^*}}\geq\frac{\delta}{2}\left(\int_{\Omega}(u_k+\delta)^{\frac{2_s^*\beta}{2}}dx \right)^{\frac{2}{2_s^*}}-\delta^{\beta+1}|\Omega|^{\frac{2}{2_s^*}}
\end{align}
Therefore, on using \eqref{bound est 6} in \eqref{bound est 5} and then applying \eqref{bound est 5} to the last inequality of \eqref{bound est 4}, we get
\begin{align}\label{bound est 7}
\left|(u_k+\delta)^{\frac{\beta}{2}}\right|^2_{2_s^*}&\leq C\left(\frac{2(\beta+1)^2}{4\beta\delta C_{N, s}}|u|_{2_s^*}^{\alpha}|(u_k+\delta)^{\beta}|_q+\delta^{\beta}|\Omega|^{\frac{2}{2_s^*}} \right)\nonumber\\
&\leq C\left(\frac{2(\beta+1)^2}{4\beta\delta C_{N, s}}|u|_{2_s^*}^{\alpha}|(u_k+\delta)^{\beta}|_q+\frac{(\beta+1)^2}{4\beta}|\Omega|^{1-\frac{1}{q}-\frac{2s}{N}}|(u_k+\delta)^{\beta}|_q \right)\nonumber\\
&\leq C \frac{(\beta+1)^2}{4\beta}|(u_k+\delta)^{\beta}|_q\left(\frac{|u|_{2_s^*}^{\alpha}}{\delta C_{N,s}}+|\Omega|^{1-\frac{1}{q}-\frac{2s}{N}} \right)\\
\text{where},~ C=C(N, s)>0.\nonumber
\end{align}
 
\noindent We will now choose, $\delta=\cfrac{|u|_{2_s^*}^{\alpha}}{\delta C_{N,s}|\Omega|^{1-\frac{1}{q}-\frac{2s}{N}}}>0$ and $\beta\geq1$ such that $\left(\frac{\beta+1}{2\beta}\right)^2\leq1.$ Let us now choose $\eta=\frac{2_s^*}{2q}>1$ and $\tau=q\beta$ and then rewrite the inequality \eqref{bound est 7} as
\begin{align}\label{bound est 8}
|(u_k+\delta)|_{\eta\tau}\leq\left(C|\Omega|^{1-\frac{1}{q}-\frac{2s}{N}} \right)^{\frac{q}{\tau}}\left(\frac{\tau}{q} \right)^{\frac{q}{\tau}}|(u_k+\delta)|_{\tau}
\end{align}
We now iterate the inequality \eqref{bound est 8} by setting the following sequence. $$\tau_0=q ~~\text{and}~~ \tau_{n+1}=\eta\tau_n=\eta^{n+1}q$$
Hence, after $n$ iteration, the inequality \eqref{bound est 8} reduces to 
\begin{align}\label{bound est 9}
\left|(u_k+\delta)\right|_{\tau_{n+1}}\leq\left(C|\Omega|^{1-\frac{1}{q}-\frac{2s}{N}} \right)^{\left(\sum\limits_{i=0}^{n}\frac{q}{\tau_i}\right)}\prod\limits_{i=0}^{n}\left(\frac{\tau_i}{q} \right)^{\frac{q}{\tau_i}}|(u_k+\delta)|_{q}
\end{align}
Since $\eta>1$, we have
$$\sum\limits_{i=0}^{\infty}\frac{q}{\tau_i}=\sum\limits_{i=0}^{\infty}\frac{1}{\eta^i}=\frac{\eta}{\eta-1}$$ and 
$$\prod\limits_{i=0}^{\infty}\left(\frac{\tau_i}{q} \right)^{\frac{q}{\tau_i}}=\eta^{\frac{\eta}{(\eta-1)^2}}.$$
Now, by taking limit $n\rightarrow\infty$ in \eqref{bound est 9}, we get
\begin{equation}\label{bound est 10}
|u_k|_{\infty}\leq\left(C|\Omega|^{1-\frac{1}{q}-\frac{2s}{N}} \right)^{\frac{\eta}{\eta-1}}\eta^{\frac{\eta}{(\eta-1)^2}}|(u_k+\delta)|_{q}
\end{equation}
Therefore, by using $u_k\leq(u-1)^+$ and the triangle inequality in \eqref{bound est 10}, we get
\begin{equation}\label{bound est 11}
|u_k|_{\infty}\leq C\eta^{\frac{\eta}{(\eta-1)^2}}\left(|\Omega|^{1-\frac{1}{q}-\frac{2s}{N}} \right)^{\frac{\eta}{\eta-1}}\left(|(u-1)^+|_q+\delta|\Omega|^{\frac{1}{q}}\right)
\end{equation}
Now letting $k\rightarrow\infty$ in \eqref{bound est 11}, we have
\begin{equation}\label{bound est 12}
|(u-1)^+|_{\infty}\leq C\eta^{\frac{\eta}{(\eta-1)^2}}\left(|\Omega|^{1-\frac{1}{q}-\frac{2s}{N}} \right)^{\frac{\eta}{\eta-1}}\left(|(u-1)^+|_q+\delta|\Omega|^{\frac{1}{q}}\right)
\end{equation}
Hence, we conclude that $u\in L^{\infty}(\Omega).$
\end{proof}

\begin{proof}[{\bf Proof of Theorem \ref{main thm}}]
	From the definition of $\xi$ and the assumption $(A1)$, we have $\tilde{I}$ is even and $\tilde{I}(0)=0.$ Therefore, Lemma \ref{sym mountain}, Lemma \ref{lemma ps} and Lemma \ref{lemma genus} guarantees that $\tilde{I}$ has sequence of critical points $\{u_n\}$ such that $\tilde{I}(u_n)<0$ and $\tilde{I}(u_n)\rightarrow0^-$.\\
	{\bf Claim:} Suppose $u_n$ is a critical point of $\tilde{I}$, then for each $n\in\mathbb{N}$, $u_n\geq0$ a.e. in $X_0$.
	\begin{proof}
		Let us consider, $\Omega= \Omega^+\cup\Omega^-$, where $\Omega^+=\{x\in X_0: u_n(x)\geq0 \}$ and $\Omega^-=\{x\in X_0: u_n(x)<0 \}$. We define $u_n(x)=u_n^+-u_n^-$, where $u_n^+(x)=\max\{u_n(x), 0\}$ and $u_n^-(x)=\max\{-u_n(x), 0\}$. Suppose, $u_n<0$ a.e. in $\Omega$, then on taking, $\phi=u_n^-$ as the test function in the equation \eqref{weak cutoff} in conjunction with the inequality $(a-b)(a^--b^-)\leq-(a^--b^-)^2$, we get
		\begin{align*}
		&\int_{\Omega}\left(\lambda\frac{sign(u_n)u_n^-}{|u_n|^{\gamma}}+\tilde{f}(x,u_n)u_n^-\right)dx=\int_{Q}\frac{(u_n(x)-u_n(y))(u_n^-(x)-u_n^-(y))}{|x-y|^{N+2s}}dxdy\\
		&\Rightarrow\lambda\int_{\Omega^-}|u_n^-|^{1-\gamma}dx\leq -\|u_n^-\|^2<0.
		\end{align*}
		This implies $|\Omega^-|=0$, which is a contradiction to the assumption $u_n<0$ a.e. in $\Omega$. This proves our claim.
	\end{proof}
	\noindent Moreover, from the definition of $\tilde{I}$, we have
	\begin{align*}
	\frac{1}{\alpha}\langle\tilde{I}^{'}(u_n), u_n\rangle-\tilde{I}(u_n)&=\frac{1}{\alpha}\left[\|u_n\|^2-\int_{\Omega}\left(\lambda\frac{sign(u_n)u_n}{|u_n|^{\gamma}}+\tilde{f}(x,u_n)u_n\right)dx\right]\\
	&\hspace{2.3cm}-\left[\frac{1}{2}\|u_n\|^2-\int_{\Omega}\left(\frac{\lambda}{1-\gamma}|u_n|^{1-\gamma}+\tilde{F}(x,u_n)\right)dx\right]\\
	&=(\frac{1}{\alpha}-\frac{1}{2})\|u_n\|^2-\lambda(\frac{1}{\alpha}-\frac{1}{1-\gamma})\int_{\Omega}|u_n|^{1-\gamma}dx\\
	&\hspace{3.1cm}+\frac{1}{\alpha}\int_{\Omega}(\alpha\tilde{F}(x,u_n)-\tilde{f}(x,u_n))dx\\
	&\geq(\frac{1}{\alpha}-\frac{1}{2})\|u_n\|^2+\lambda(\frac{1}{1-\gamma}-\frac{1}{\alpha})\int_{\Omega}|u_n|^{1-\gamma}dx\\
	&\geq(\frac{1}{\alpha}-\frac{1}{2})\|u_n\|^2
	\end{align*}
	It can be easily seen that, since
	\begin{align*}
	&\frac{1}{\alpha}\langle\tilde{I}^{'}(u_n), u_n\rangle-\tilde{I}(u_n)=o_n(1)\\
	&\Rightarrow(\frac{1}{\alpha}-\frac{1}{2})\|u_n\|^2\leq o_n(1),
	\end{align*}
	hence, we get $u_n\rightarrow0$ in $X_0.$ Now by using Moser iteration and Lemma \ref{bounded}, we can assume that as $n\rightarrow\infty$, $\|u_n\|_{L^{\infty}(\Omega)}\leq l.$ Therefore, the problem \eqref{main2} has infinitely many solutions. Further, due to the nonnegativity of $u_n$ and $\tilde{I}(u_n)<0$, we conclude that the problem \eqref{main1} has infinitely many weak solutions and hence the Theorem \ref{main thm} is proved.
\end{proof}
\begin{remark}
	$\Lambda>0$, because the solution to the problem \eqref{main1} exists for some $\lambda>0.$
\end{remark}

\section*{Acknowledgement}
The author S. Ghosh, thanks the Council of Scientific and Industrial Research (C.S.I.R), India, for the financial assistantship received to carry out this research work. Both the authors thanks the research facilities received from the Department of Mathematics, National Institute of Technology Rourkela, India. The authors thank the anonymous reviewers for their constructive comments and suggestions.

\bibliographystyle{plain}

\end{document}